\documentclass[11pt]{article}

\usepackage[letterpaper,margin=1in]{geometry}

\newcommand{\triplenorm}[1]{
  \left\vert\kern-0.9pt\left\vert\kern-0.9pt\left\vert #1
  \right\vert\kern-0.9pt\right\vert\kern-0.9pt\right\vert}  
\newcommand{\Tscalar}[2]{\ensuremath{(#1 , #2)_{\bar R^{-1}}}} 
\newcommand{\Tnorm}[1]{\ensuremath{\|#1\|_{\bar R^{-1}}}}
\newcommand{\Dscalar}[2]{\ensuremath{( #1 , #2 )_D}}
\newcommand{\Dnorm}[1]{\|#1\|_{D}}

\newcommand{\Dproj}{\ensuremath{Q_D}}

\newcommand{\eqqsim}{\mathbin{\rotatebox[origin=c]{180}{\ensuremath{\cong}}}}

\newcommand{\co}{\ensuremath{C_{o}}}

\usepackage{amsmath}
\usepackage{amsbsy}
\usepackage{amsfonts}
\usepackage{amssymb}
\usepackage{amscd}
\usepackage{mathrsfs}

\usepackage{bm}

\usepackage{algpseudocode}
\usepackage[Algorithm]{algorithm}

\algnewcommand{\IIf}[1]{\State\algorithmicif\ #1\ \algorithmicthen}
\algnewcommand{\EElse}{\unskip\ \algorithmicelse\ }
\algnewcommand{\EndIIf}{\unskip\ \algorithmicend\ \algorithmicif}
\algnewcommand{\FFor}[1]{\State\algorithmicfor\ #1\ }
\algnewcommand{\EndFFor}{\unskip\ \algorithmicend\ \algorithmicfor}

\usepackage{txfonts}

\usepackage[latin1]{inputenc}

\usepackage{graphicx}
\usepackage{subfigure} 
\usepackage{float}

\usepackage{tikz}   
\usetikzlibrary{shapes,arrows,decorations.pathmorphing,backgrounds,positioning,fit,matrix,calc}

\tikzstyle{decision} = [diamond, draw, fill=blue!20, 
    text width=4.5em, text badly centered, node distance=3cm, inner sep=0pt]
\tikzstyle{block} = [rectangle, draw, fill=blue!20, 
    text width=5em, text centered, rounded corners, minimum height=4em]
\tikzstyle{line} = [draw, -latex']
\tikzstyle{cloud} = [draw, ellipse,fill=red!20, node distance=3cm,
    minimum height=2em]

\usepackage{hyperref}
\usepackage{type1cm}       

\usepackage{soul}

\usepackage{url}
\usepackage{undertilde}

\usepackage{rotating} 
\usepackage{makeidx}
\makeindex             
\usepackage{multicol}   
\usepackage{enumerate}
\usepackage{xspace}
 \usepackage{comment}  
\usepackage{etoolbox}

\usepackage{fancyhdr}
\ifcsmacro{theorem}{}{
\newtheorem{theorem}{Theorem}[section]
}
\ifcsmacro{lemma}{}{
\newtheorem{lemma}[theorem]{Lemma}
}
\ifcsmacro{corollary}{}{
\newtheorem{corollary}[theorem]{Corollary}
}
\ifcsmacro{proposition}{}{

}
\ifcsmacro{algorithm}{}{
\newtheorem{algorithm}[equation]{Algorithm}
}

\newtheorem{assumption}[equation]{Assumption}

\let\oldchapter\chapter
\def\chapter{
  \setcounter{exercise}{0}
  \oldchapter
}

\newcommand{\beas}{\begin{eqnarray*}}
\newcommand{\eeas}{\end{eqnarray*}}
\newcommand{\bary}{\begin{array}}
\newcommand{\eary}{\end{array}}

\def\ec{\mathrel{\hbox{$\copy\Ea\kern-\wd\Ea\raise-3.5pt\hbox{$\sim$}$}}}
\newcommand{\lc}{\mathrel{\raise2pt\hbox{${\mathop<\limits_{\raise1pt\hbox{\mbox{$\sim$}}}}$}}}
\newcommand{\gc}{\mathrel{\raise2pt\hbox{${\mathop>\limits_{\raise1pt\hbox{\mbox{$\sim$}}}}$}}}

\newbox\Ea
\setbox\Ea=\hbox{\raise0.9pt\hbox{$=$}}

\newcommand{\bproof}{\begin{proof}}
\newcommand{\eproof}{\end{proof}}

\ifcsmacro{R}{}{
 
}

\newcommand{\diam}{\mbox{\rm diam\,}}

\newtheorem{remark}[theorem]{Remark}
\newtheorem{definition}[theorem]{Definition}

\numberwithin{equation}{section}

\begin{document}

\title{A unified approach to the design and analysis of AMG}
\author{Jinchao Xu, Hongxuan Zhang and Ludmil Zikatanov}

\maketitle

\begin{abstract}
  In this work, we present a general framework for the design and
  analysis of two-level AMG methods. The approach is to find a basis
  for locally optimal or quasi-optimal coarse space, such as the space
  of constant vectors for standard discretizations of scalar elliptic
  partial differential equations. The locally defined basis elements
  are glued together using carefully designed linear extension maps to
  form a global coarse space. Such coarse spaces, constructed locally,
  satisfy global approximation property and by estimating the local
  Poincar{\' e} constants, we obtain sharp bounds on the convergence
  rate of the resulting two-level methods. To illustrate the use of
  the theoretical framework in practice, we prove the uniform
  convergence of the classical two level AMG method for finite element
  discretization of a jump coefficient problem on a shape regular
  mesh.
\end{abstract}

\section{Introduction}

Multigrid methods are among the most efficient numerical methods for
solving large scale linear systems of equations arising from the
discretization of partial differential equations.  This type of
methods can be viewed as an acceleration of traditional iterative
methods based on local relaxation such as Gauss-Seidel and Jacobi
methods.  The main idea behind multigrid methods is to project the
error obtained after applying a few iterations of local relaxation
methods onto a coarser grid. Part of the slow-to-converge (the low
frequency) error on a finer grid is a relatively high frequency on the
coarser grid and such high frequencies can be further corrected by a
local relaxation method on the coarser grid.  By recursively repeating
such a procedure a multilevel iterative process is obtained.  A
classical example of a multilevel algorithm is known as Geometric
Multi-Grid (GMG) method, which converges uniformly with nearly
optimal complexity for a large class of problem, especially elliptic
boundary problems of 2nd and 4th order as demonstrated
in~\cite{1975NicolaidesR-aa,1977NicolaidesR-aa,1980BankR_DupontT-aa,1983BraessD_HackbuschW-aa,bramble1987new,bramble1990parallel,1991BrambleJ_PasciakJ_XuJ-aa,1991BrambleJ_PasciakJ_WangJ_XuJ-ac,1992XuJ-aa}.

Despite of their efficiency, however, the GMG methods have their
limitations. They depend on a hierarchy of geometric grids which is
often not readily available.  The Algebraic MultiGrid (AMG) methods
were designed in an attempt to address such limitations. They were
proposed as means to generalize geometric multigrid methods for
systems of equations that share properties with discretized PDEs, such
as the Laplace equation, but use unstructured grids in
the underlying discretization.  The first AMG algorithm
in~\cite{1stAMG} was a method developed
under the assumption that such a problem was being solved.  Later, the
AMG algorithm was generalized using many heuristic to extend its
applicability to more general problems and matrices.  As a result, a
variety of AMG methods have been developed in the last three decades
and they have been applied to many practical problems with
success. But, unfortunately, a good theoretical understanding of why
and how these methods work is still seriously lacking.

One of the first results on two level convergence of AMG methods are
found in earlier papers~\cite{1stAMG,Ruge.J;Stuben.K.1987a}.  There
have been a lot of research on reflecting the MG theory through algebraic settings:
\cite{1983MaitreJ_MusyF-aa,1985BankR_DouglasC-aa,Mandel.J.1988a}; algebraic
variational approach to the two level MG
theory~\cite{1985McCormickS-aa,1984McCormickS-aa,1982McCormickS_RugeJ-aa}. 
For the two grid convergence, sharper results, including two sided
bounds are given in~ \cite{2008ZikatanovL-aa} and also considered in
\cite{Falgout.R;Vassilevski.P.2004a}
and~\cite{Falgout.R;Vasilevski.P;Zikatanov.L.2005a}. These two-level
results are more or less a direct consequences of the abstract theory
provided
in~\cite{1991BrambleJ_PasciakJ_WangJ_XuJ-ac,1992XuJ-aa,Xu.J;Zikatanov.L.2002a}. A
survey of these and other related results is found in a recent
article~\cite{2014MacLachlanS_OlsonL-aa}.

Multilevel results are difficult to establish in general algebraic
settings, and most of them are based on either not realistic
assumptions or they use geometrical grids to prove convergence. We
refer
to~\cite{1996VanekP_MandelJ_BrezinaM-aa,2011BrezinaM_VassilevskiP-aa}
for results in this direction. Rigorous multilevel results for finite element
equations can be derived using the auxiliary space framework, which is
developed in~\cite{1996XuJ-aa} for quasi-uniform meshes. More recently
multilevel convergence results for adaptively refined grids were shown
to be optimal in~\cite{2012ChenL_NochettoR_XuJ-aa}.  A multilevel
convergence result on shape regular grids using AMG based on quad-tree
(in 2D) and oct-tree (in 3D) coarsening is shown
in~\cite{2015GrasedyckL_WangL_XuJ-aa}.

In this paper, we focus on the design and analysis of the two level
AMG methods.  We develop a unified framework and theory that can be
used to derive and analyze different algebraic multigrid methods in a
coherent manner. We provide a general approach to the construction of
coarse space and we prove that under appropriate assumptions the
resulting two-level AMG method for the underlying linear system
converges uniformly with respect to the size of the problem, the
coefficient variation, and the anisotropy.  Our theory applies to most
existing multigrid methods, including the standard geometric multigrid
method~\cite{1992XuJ-aa,1978HackbuschW-aa}, the classic
AMG~\cite{1stAMG}, energy-minimization AMG~\cite{energymin},
unsmoothed and smoothed aggregation AMG~\cite{Vakhutinsky.I;Dudkin.L;Ryvkin.A.1979a, Mika.S;Vanek.P.1992a,Mika.S;Vanek.P.1992b,Napov.A;Notay.Y.2012a,Notay.Y.2012b}, and
spectral AMGe~\cite{Jones.J;Vassilevski.P.2001a,Brezina.M;Cleary.A;Falgout.R;Henson.V;Jones.J;Manteuffel.T;McCormick.S;Ruge.J.2001a,Chartier.T;Falgout.R;Henson.V;Jones.J;Manteuffel.T;McCormick.S;Ruge.J;Vassilevski.P.2003b}. As an application, we prove, using
our abstract framework, the uniform convergence of the standard two-level
classical AMG method for jump coefficient problem.

With very few exceptions, the AMG algorithms have been mostly
targeting the solution of symmetric positive definite (SPD) systems.
In this paper, we choose to present our studies for a slightly larger
class of problems, namely symmetric semi-positive definite (SSPD)
systems.  This approach is not only more inclusive, but more
importantly, the SSPD class of linear systems can be viewed as more
intrinsic to the AMG ideas.  For example, the design of AMG may be
better understood by using local problem (defined on subdomains) with
homogeneous Neumann boundary condition, which would amount to an SSPD
sub-systems.

In short, in this paper we consider AMG techniques for solving a linear
system of equations: 
\begin{equation}
  \label{Au=f}
	Au =f,
\end{equation}
where $A$ is a given SSPD operator or sparse matrix, and the problem
is posed in a vector space of a large dimension. Furthermore, we will 
show in \S\ref{sec:m-matrix} and \S\ref{sec:m-matrix-fem} that in most
cases, $A$ can be replaced by an $M$-matrix, which we call it \emph{$M$-matrix 
relative} of $A$.

\section{Model elliptic PDE operators}\label{sec:model}
We consider the following boundary value problems
\begin{equation}
  \label{Model0}
    {\mathcal L}u=-\nabla\cdot (\alpha(x)\nabla u)=f, \quad x\in \Omega
\end{equation}
where $\alpha: \Omega\mapsto \mathbb R^{d\times d}$ is an SPD matrix function
satisfying
\begin{equation}
  \label{alpha}
\alpha_0\|\xi\|^2\le 
\xi^T\alpha (x)\xi \le
\alpha_1\|\xi\|^2,\quad \xi\in \mathbb R^d.
\end{equation}
for some positive constants $\alpha_0$ and $\alpha_1$.  Here $d=1,2,3$
and $\Omega\subset\mathbb R^d$ is a bounded domain with boundary
$\Gamma=\partial \Omega$.

A variational formulation for \eqref{Model0} is as follows: Find
$u\in V$ such that
\begin{equation}
  \label{Vari}
a(u,v)=(f, v), \quad\forall v\in V.   
\end{equation}
Here
$$
a(u,v)=\int_\Omega (\alpha(x)\nabla u)\cdot \nabla v, \quad 
(f,v)= \int_\Omega fv.
$$ 
and $V$ is a Sobolev space that can be chosen to address different
boundary conditions accompanying the equation~\eqref{Model0}. 
One case is the mixed boundary conditions: 
\begin{equation}
  \label{MixedBoundary}
  \begin{array}{rcl}
u=&0, &x\in \Gamma_D,\\
(\alpha\nabla u)\cdot n=&0,&x \in\Gamma_N,
\end{array}
\end{equation}
where
$\Gamma=\Gamma_D\cup\Gamma_N$.  The pure Dirichlet problem is when
$\Gamma_D = \Gamma$, namely
\begin{equation}\label{dirichlet}
    u= 0, \quad x\in \Gamma,
\end{equation}
while the pure Neumann problem is when $\Gamma_N
=\Gamma$, namely
\begin{equation}\label{neumman}
    (\alpha\nabla u)\cdot n = 0 , \quad x\in \Gamma.
\end{equation}
We thus have $V$ as
\begin{equation}
  \label{3V}
V=
\left\{
  \begin{array}{l}
    H^1(\Omega) = \{v\in L^2(\Omega): \partial_iv\in    L^2(\Omega), i=1:d\};\\
    H^1_D(\Omega) = \{v\in H^1(\Omega): v|_{\Gamma_D}=0\}. 
  \end{array}
\right.  
\end{equation}
When we consider a pure Dirichlet problem, $\Gamma_D = \Gamma$, we
denote the space by $V=H^1_0(\Omega)$. In addition, for pure Neumann
boundary conditions, the following condition is  added to
assure the existence of the solution to~\eqref{Vari}:
\begin{equation}
  \label{consistent-f}  \int_\Omega f =0. 
\end{equation}

One most commonly used model problem is when 
\begin{equation}
  \label{iso}
\alpha(x)=1, \quad x\in \Omega,
\end{equation}
which corresponds to the Poisson equation
\begin{equation}
  \label{Poisson}
-\Delta u=f.  
\end{equation}
This simple problem provides a good representative model for isotropic problems.

In this paper, we focus on the special case 
when $\alpha$ is a scalar and it has discontinuous jumps such as
\begin{equation}\label{coeff12}
\alpha(x) = \begin{cases}
\epsilon , \quad x\in \Omega_1,\\
1 , \quad x\in \Omega_2.
\end{cases}
\end{equation}
The interesting jump coefficient case is when $\epsilon$ is sufficiently small, and we make such an assumption to investigate the robustness of algorithms with respect to
the PDE coefficient variation. 

We now give an example of finite element discretization. 
Given a triangulation ${\mathcal T}_h$ for $\Omega$, let $V_h\subset V$ be a finite
element space consisting of piecewise linear (or higher order)
polynomials with respect to the triangulation ${\mathcal T}_h$.  The
finite element approximation of the variational problem \eqref{Vari}
is: Find $u_h\in V_h$ such that
\begin{equation}
  \label{vph}
a(u_h,  v_h)=(f,v_h), \quad\forall\,v_h\in V_h.
\end{equation}
Assume $\{\phi_i\}_{i=1}^{N}$ is the nodal basis of $V_h$,  namely,
$\phi_i(x_j)=\delta_{ij}$ for any nodes $x_j$.
We write
\(
        u_h(x)=\sum_{j=1}^{N}\mu_j\phi_j(x)
\)
the equation \eqref{vph}  is then equivalent to
\[
        \sum_{j=1}^{N}\mu_ja(\phi_j,\phi_i)=(f,\phi_i),\quad
        j=1,2,\cdots, N, 
\]
which is a linear system of equations:
\begin{equation}\label{axb}
        A\mu=b, \quad (A)_{ij} = a(\phi_j,\phi_i), \quad 
 \mbox{and}\quad (b)_i=(f,\phi_i).
\end{equation}
Here, the matrix $A$ is known as the stiffness matrix of the  nodal basis
$\{\phi_i  \}_{i=1}^N$.

For any $T\in \mathcal T_h$, we define
\begin{equation}\label{hT}
    \overline h_T=\diam(T),\quad h_T=|T|^{\frac{1}{d}}, \quad \underline h_T=2\sup\{r>0: B(x, r)\subset T \text{ for } x\in T\}.
\end{equation}
We say that the mesh $\mathcal T_h$ is \emph{shape regular} if there exists a uniformly bounded constant $\sigma \ge 1$ such that     
\begin{equation}
    \underline h_T\le h_T \le \overline h_T\le \sigma \underline h_T, \quad \forall T\in \mathcal T_h.
\end{equation}
And we call  $\sigma$ the \emph{shape regularity constant}. 

In the following we assume that the finite element mesh is shape regular.

\section{An abstract two-level method}

Given a finite dimensional vector space $V$ equipped with an inner product
$(\cdot,\cdot)$, we consider
\begin{equation}
\label{Auf}
Au=f,   
\end{equation}
where $ A: V\mapsto V' $ is symmetric positive definite (SPD)
and $V'$ is the dual of $V$. 

A two-level method for solving \eqref{Auf} typically consists of the following components:
\begin{enumerate}[1.]
\item A smoother $R: V'\mapsto V$;
\item A coarse space $V_c\subset V$ linked with $V$ via a prolongation operator:
$$
P:  V_c\mapsto V. 
$$
\item A coarse space solver $B_c: V_c'\mapsto V_c$.
\end{enumerate}

We always assume that $\bar R$ is SPD and hence the smoother $R$ is always convergent. 
Furthermore, 
\begin{equation}
    \|v\|_A^2\le \|v\|_{\bar R^{-1}}^2.
\end{equation}
In the discussion below we need the following inner product
\begin{equation}\label{eq:norm-star}
\Tscalar{u}{v} = (\overline{T}^{-1} u,v)_A=( \bar R^{-1}u,v), \quad \overline{T} = \overline{R} A,
\end{equation}
and the accompanying norm $\Tnorm{\cdot}$. 

The restriction of \eqref{Auf} is then
\begin{equation}
  \label{coarse:Au=f}
A_cu_c=f_c,   
\end{equation}
where 
$$
A_c=P'AP, \quad f_c=P'f. 
$$

In an axact two-level method, the coarse space solver $B_c$ is chosen to be the exact solver,
namely $B_c=A_c^{-1}$. In the case that $A$ is semi-definite, we use $N(A)$ to denote the kernel of $A$ and we always assume that $N(A)\subset V_c$. When $N(A)\neq \{0\}$ 
with a slight abuse of notation, we will still use $A_c^{-1}$ to denote the psudo-inverse of
$A_c$, and in such case we have
$$
A_c^{-1}=A_c^\dag
$$
We will use similar notation for psudo-inverse of other relevant singular operators and
matrices in the rest of the paper. 

A typical AMG algorithm is defined in terms of an operator
$B: V'\mapsto V$, which is an approximate inverse (a preconditioner)
of $A$.  The two level MG method is as follows.
\begin{algorithm}[H]
\caption{A two level MG method}\label{alg:two-level}
Given $g\in V'$ the action  $Bg$  is defined via the following three steps
\begin{enumerate}
\item Coarse grid correction: $w=P B_c P' g$.
\item Post-smoothing: $Bg:= w +R(g-Aw)$.
\end{enumerate}
\end{algorithm}
    
The error propagation operator for two-level AMG operator $E=I-BA$ is
    \begin{equation}\label{E_op}
        E=(I-RA)(I-\Pi_c), 
    \end{equation}
    where $\Pi_c= PA_c^{-1}P^TA$, which is the $(\cdot,\cdot)_A$ orthogonal projection on $V_c$.

    The following convergence result is shown in~\cite{AMGReview} for
    semi-definite operators $A$ and is an improvement of the
    well-known two level convergence estimates considered
    in~\cite{Xu.J;Zikatanov.L.2002a,Falgout.R;Vasilevski.P;Zikatanov.L.2005a}.

\begin{theorem}\label{thm:two-level-convergence}  
Assume that $N\subset V_c$. The convergence rate of an exact two level
AMG is given by 
\begin{equation}\label{eq:two-level-convergence}
\|E\|_A^2 = 1- \frac1{K(V_c)}, 
\end{equation}
where
\begin{equation}
  \label{KVc}
K(V_c)=\max_{v\in V}\min_{v_c\in V_c}\frac{\Tnorm{v-v_c}^2}{\|v\|^2_A}.
\end{equation}
\end{theorem}

For a given smoother $R$, one basic strategy in the design of AMG is
to find a coarse space such that $K(V_c)$ is made as practically small
as possible.  There are many cases, however, in which the operator
$\bar{R}^{-1}$ in the definition of $K(V_c)$ is difficult to work
with.  It is then convenient to replace $\bar{R}^{-1}$ by a
simpler and spectrally equivalent SPD operator.  More specifically, we
assume that $D: V\mapsto V'$ is an SPD operator such that
\begin{equation}\label{star-equiv-norms} c_D \Dnorm{v}^2\le
\Tnorm{v}^2 \le c^D\Dnorm{v}^2, \quad \forall v\in V,
\end{equation} where
\[ \Dscalar{u}{v}=( Du,v), \quad \Dnorm{v}^2 =
\Dscalar{v}{v}.
\]

As a rule, the norm defined by $\bar{R}$ corresponding to the
symmetric Gauss-Seidel method, i.e. $R$ defined by pointwise
Gauss-Seidel method can be replaced by the norm defined by the
diagonal of $A$ (i.e. by Jacobi method, which, while not always
convergent as a relaxation provides an equivalent norm).  For additional
details on this equivalence, we refer to~\cite{2008ZikatanovL-aa}.

Now, in terms of this operator $D$, we introduce the following quantity
\begin{equation}
  \label{KVcD} K(V_c,D)= \max_{v}\frac{\Dnorm{v-\Dproj
v}^2}{\|v\|_A^2} =\max_{v}\min_{v_c\in
V_c}\frac{\Dnorm{v-v_c}^2}{\|v\|_A^2},
\end{equation} 
where $\Dproj: V\mapsto V_c$ is the $\Dscalar{u}{v}$-orthogonal projection.
By \eqref{KVc}, \eqref{KVcD} and \eqref{star-equiv-norms}, we have
\begin{equation}
  \label{KK} c_D K(V_c,D)\le K(V_c)\le c^D K(V_c,D).
\end{equation}
The following theorem presents the two sided bounds on the convergence rate of the two level methods depending on the constants involved in~\eqref{star-equiv-norms}.
\begin{theorem}\label{thm:two-level-theorem-period} The two level
algorithm satisfies
\begin{equation}\label{eq:y-two-level-estimate}
1-\frac{1}{c_DK(V_c,D)}\le \|E\|_A^2 \le 1-\frac{1}{c^DK(V_c,D)} \le
1-\frac{1}{c^DC}.
\end{equation} where $C$ is any upper bound of $K(V_c,D)$, namely
\begin{equation}\label{eq:approximation-z} 
    \min_{w\in V_c}\Dnorm{v-w}^2 \le C\|v\|_A^2, \quad\mbox{for all}\quad v\in V.
\end{equation}
\end{theorem} The proof of the above theorem is straightforward and indicates that, if $c_D$ and $c^D$ are ``uniform'' constants,
the convergence rate of the two-level method is ``uniformly'' dictated
by the quantity $K(V_c,D)$.

\section{$M$-matrix relatives}\label{sec:m-matrix}

Our results on $M$-matrix relatives are related to the some of the
works on preconditioning by $Z$-matrices and
$L$-matrices~\cite{kraus2006algebraic,kraus2008algebraic}. They are
implicitly used in most of the AMG
literature~\cite{Ruge.J;Stuben.K.1987a} where the classical strength
of connection definition gives an $M$-matrix.

In this paper, a symmetric matrix $A \in \mathbb R^{n\times n}$ is called an \emph{$M$-matrix} if it satisfies
the following three properties:
\begin{align}
\label{eq:sign1} &a_{ii} > 0 \;\;\text{for}\;\; i = 1,...,n,\\
\label{eq:sign2} &a_{ij} \le 0 \;\;
                  \text{for}\;\; i \ne j, \;\;i, \;j = 1,...,n,\\
\label{eq:sign3} &A \;\;\text{is semi-definite}.
\end{align}
An important remark is in order: We have used the term $M$-matrix to
denote semidefinite matrices, and we are aware that this is not the
precise definition. It is however convenient to use reference to
$M$-matrices and we decided to relax a bit the definition here with
the hope that such an inaccuracy pays off by better appeal to the
reader.
 
As first step in creating space hierarchy the majority of the AMG
algorithms for $Au = f$ with positive semidefinite $A$ uses a simple
filtering of the entries of $A$ and construct an $M$-matrix which is
then used to define crucial AMG components. We next define such $M$-matrix relative. 
\begin{definition}[$M$-matrix relative] We call a matrix
  $A_M$ an \emph{$M$-matrix relative} of $A$ if
  $A_M$  is an M-matrix and satisfies the inequalities
\begin{equation}\label{MMrel-0}
(v,v)_{A_M} \lesssim (v,v)_{A}, \quad\mbox{and}\quad
(v,v)_{D}\lesssim (v,v)_{D_M}, \quad \mbox{for all} \quad v\in V,
\end{equation}
where $D_M$ and $D$ are the diagonals of $A_M$ and $A$ respectively. 
\end{definition}
We point out that the $M$-matrix relatives are instrumental in the
definition of coarse spaces and also in the convergence rate
estimates. This is clearly seen later in~\S\ref{sec:unifiedAMG} where
we present the unified two level theory for AMG.  Often, we have that
the one sided inequality in~\eqref{MMrel-0} is in fact a spectral
equivalence.

By definition, we have the following simple but important result. 
\begin{lemma}\label{lemma-equiv}
Let $A_M$ be an
  $M$-matrix relative of $A$ and let $D$ and $D_M$ be the diagonal matrices of
  $A$ and $A_M$, respectively. If $V_c\subset V$ is a subspace, then the estimate
  \begin{equation}
    \label{VcA}
\|u - u_c\|_D^2\lesssim \|u\|_A^2    
  \end{equation}
holds for some $u_c\in V_c$, if the estimate 
\begin{equation}
  \label{VcA+}
\| u - u_c \|_{D_M}^2\lesssim \|u\|_{A_M}^2  
\end{equation}
holds. 
\end{lemma}
This result, combining with the two-level convergence result,  means that we only need to work on the M-matrix relative
of $A$ in order to get the estimate \eqref{VcA}. 

We next describe how to construct M-matrix relatives for a special
class of matrices.  We first prove an auxiliary result for a special
class of matrices defined via bilinear forms
\begin{equation}\label{b-form}
(A_bu,v) : = b(u,v) = \sum_{e\in \mathcal{E}_b} \omega_e (\delta_e u)(\delta_e v).
\end{equation}
Here $\mathcal{E}_b$ is the set of edges of a connected graph with
vertices $\{1,\ldots,k\}$ and $b(\cdot,\cdot)$ is the bilinear form corresponding
to a weighted graph Laplacian.  The other quantities in~\eqref{b-form} are defined as follows:
\[
\mbox{For $e\in \mathcal{E}_b$, $e=(i,j)$, we set $\delta_e u=(u_i-u_j)$}
\]

Some of the weights in
$b(\cdot,\cdot)$ may be negative, but they should not dominate: we
assume that $b(\cdot,\cdot)$ is positive semidefinite with one
dimensional kernel spanned by $(1,\ldots,1)^T$.  If the weights
$\omega_e$ were positive then it is easy to show that this assumption
holds. Indeed, the bilinear form is obviously semidefinite and the
second part of the assumption follows from the fact that the graph is
connected.  Thus, there exists a $\lambda_b>0$ such that for all
$u\in \mathbb{R}^k$ satisfying $\sum_{i=1}^k u_i = 0$ we have
\begin{equation}\label{algebraic-connectivity}
\lambda_b \|u\|_{\ell^2}^2 \le b(u,u).
\end{equation}
Let us now denote
\[
\mathcal{E}_b^+ = \{e\in \mathcal{E}_b\;\big|\; \omega_e > 0\}, \quad
\mathcal{E}_b^- = \{e\in \mathcal{E}_b\;\big|\; \omega_e \le 0\}. 
\]
and then split the bilinear form $b(\cdot,\cdot)$ in positive and
negative parts:
\begin{eqnarray}
&&b(u,v) = b_{+}(u,v) - b_{-}(u,v),\\
&& (A_{b,+}u,v)=b_+(u,v) = \sum_{e\in \mathcal{E}_b^+} \omega_e \delta_e u\delta_e v, 
\label{bplus}
\\
&&b_-(v,v) =  \sum_{e\in \mathcal{E}_b^-} |\omega_e| \delta_e u\delta_e v.\label{b-min} 
\end{eqnarray}
We observe that $A_b$ defined via the bilinear form $b(\cdot,\cdot)$
in~\eqref{b-form} is an $M$-matrix relative to itself if
$\mathcal{E}_b^{-} = \emptyset$, or, equivalently, $\mathcal{E}_b^+=\mathcal{E}_b$.

The following lemma gives an estimate of $b_-(\cdot,\cdot)$ and
$b_+(\cdot,\cdot)$ in terms of $b(\cdot,\cdot)$ basically showing that
$A_{b,+}$ may be used as $M$-matrix relative to $A_b$.
\begin{lemma}\label{MMrel} 
  If $\omega_- = \max_{e\in \mathcal{E}_b^-}|\omega_e|$ then we have the
  following inequalities for all $v\in \mathbb{R}^k$,
\begin{eqnarray}
&& b(v,v) \le b_+(v,v)\le \left(1+\frac{c(k)\omega_{-}}{\lambda_b}\right)b(v,v),
\label{e:ineq1}\\
&&\|v\|_{D_b}^2 \le \|v\|_{D_b^+}^2\le 
\left(1+\frac{c(k)\omega_{-}}{\lambda_b}\right)\|v\|_{D_b}^2. 
\label{e:ineq2}
\end{eqnarray}
where $D_b$ and $D_{b}^+$ are the diagonals of 
$A_b$ and $A_b^+$ and $c(k) = 2(k-1)$. 
\end{lemma}
\begin{proof}
We first show the inequality
\begin{equation}
\label{e:MMrel}
 b_-(v,v)\le \frac{c(k)\omega_{-}}{\lambda_b}b(v,v).
\end{equation}
Clearly, we only need to consider all $v\in \mathbb{R}^k$ such that
$\sum_{i=1}^k v_i = 0$, because adding a multiple of $(1,\ldots,1)^T$
to $v$ does not change either side of the
inequality~\eqref{e:MMrel}. For such $v\in \mathbb{R}^k$ we
have
\begin{eqnarray*}
b_{-}(v,v) &\le & 
\omega_- \sum_{e\in \mathcal{E}_b^-} (\delta_e v)^2 
 \le  \omega_- \sum_{e\in \mathcal{E}_b} (\delta_e v)^2 \\
&\le & 2\omega_- \sum_{(i,j)\in \mathcal{E}_b} (v_i^2+v_j^2) 
\le 2(k-1)\omega_- \|v\|_{\ell^2}^2 \\
& \le & c(k) \omega_- \|v\|^2_{\ell^2}  \le  \frac{c(k) \omega_-}{\lambda_b} b(v,v),
\end{eqnarray*}

where we have used that a vertex $i$ cannot be in more than $(k-1)$ edges. 
This shows~\eqref{e:ineq1} after some obvious algebraic
manipulations.

The inequality~\eqref{e:ineq2} follows from~\eqref{e:ineq1}. For $i=1,2,\ldots,k$ we have
\begin{eqnarray*}
&&(D_b)_{ii} = b(\phi_i,\phi_i) \le 
b_+(\phi_i,\phi_i) =
(D_b^+)_{ii},\\
&&(D_b^+)_{ii} = b_+(\phi_i,\phi_i) \le 
\left(1+\frac{c(k)\omega_{-}}{\lambda_b}\right) 
 b(\phi_i,\phi_i) =
\left(1+\frac{c(k)\omega_{-}}{\lambda_b}\right) 
 (D_b)_{ii}.
\end{eqnarray*}
The proof is complete. 
\end{proof}

\section{$M$-matrix relatives of finite element stiffness matrices}\label{sec:m-matrix-fem}
In this section we show how to construct $M$-matrix relative to the
matrix resulting from a finite element discretization of the model
problem~\eqref{Model0} with linear elements.  We consider first an
isotropic problem with pure Neumann boundary condition \eqref{neumman}
and isotropic coefficient $\alpha(x)=a(x) I$, $x\in \Omega$. 

In the rest of this section, we make the following
assumptions on the coefficient and the geometry of $\Omega$:
\renewcommand{\labelenumi}{\arabic{enumi}.~}
\begin{enumerate}
\item The domain $\Omega\subset \mathbb{R}^d$ is partitioned into simplices
$\Omega=\cup_{T\in\mathcal{T}_h} T$.  

\item The coefficient $a(x)$ is a scalar valued function and its
  discontinuities are aligned with this partition.
\item We consider the Neumann problem, and, therefore, the bilinear
  form~\eqref{Vari} is
\begin{equation}\label{auv}
\int_{\Omega}a(x) \nabla v \cdot \nabla u  = 
\sum_{(i,j)\in \mathcal{E}} (-a_{ij})\delta_eu\delta_e v 
=\sum_{e\in \mathcal{E}} \omega_e\delta_eu\delta_e v
\end{equation}
\item 
It is well known~\cite{EAFE} that the off-diagonal entries of the stiffness matrix
$A$ are given by
\begin{eqnarray*}
&& \omega_e  =  -(\phi_j,\phi_i)_A = \sum_{T\supset e}\omega_{e,T}\\ 
&& \omega_{e,T} = \frac{1}{d(d-1)}\overline{a}_T |\kappa_{e,T}|\cot\alpha_{e,T}, \quad
\overline a_T= \frac{1}{|T|}\int_{T} a(x) \; dx
\end{eqnarray*}
Here, $e=(i,j)$ is a fixed edge with end points $x_i$ and $x_j$;
$T\supset e$ is the set of all elements containing $e$;
$|\kappa_{e,T}|$ is the volume of $(d-2)$-dimensional simplex opposite to $e$ in
$T$; $\alpha_{e,T}$ is the dihedral angle between the two faces in $T$ not containing $e$.
\item Following the notation used in Lemma~\ref{MMrel}, let
  $\mathcal{E}$ denote the set of edges in the graph defined by the
  triangulation and let $\mathcal{E}^{-}$ be the set of edges where
  $a_{ij}\ge 0$, $i\neq j$. The set complementary to $\mathcal{E}^-$ is
  $\mathcal{E}^+=\mathcal{E}\setminus\mathcal{E}^-$.  Then, with
  $\omega_e = -a_{ij}$, we have 
\begin{equation}\label{auv-do}
\int_{\Omega}a(x) \nabla v \cdot \nabla u  = 
\sum_{e\in \mathcal{E}^+} \omega_e\delta_eu\delta_ev-\sum_{e\in \mathcal{E}^-} |\omega_e|\delta_eu\delta_ev. 
\end{equation}

\item We also assume that the partitioning is such that the constant
  function is the only function in the null space of the bilinear
  form~\eqref{auv}.  This is, of course, the case when $\Omega$ is
  connected (which is true, as $\Omega$ is a domain).

\end{enumerate}

The non-zero off-diagonal entries of $A$ may have either positive or
negative sign, and, usually $\mathcal{E}^-\neq \emptyset$. The next
theorem shows that the stiffness matrix $A$ defined via the bilinear
form~\eqref{auv} is spectrally equivalent to the matrix $A_+$ defined
as
\begin{equation}\label{diag-compensate}
(A_+ u,v) = \sum_{e\in \mathcal{E}^+} \omega_e (u_i-u_j)(v_i-v_j). 
\end{equation}
Thus, we can ignore any positive off-diagonal entries in $A$, or
equivalently, we may drop all $\omega_e$ for $e\in \mathcal{E}^-$.
Indeed, $A_+$ is obtained from $A$ by adding to
the diagonal all positive off diagonal elements and setting the
corresponding off-diagonal elements to zero.  
This is a stronger result that we need later in the convergence theory
because it gives not only the inequalities~\eqref{MMrel-0} 
but also a spectral equivalence with the $M$-matrix relative $A_+$.
\begin{theorem}\label{m-matrix-plus}
  If $A$ is the stiffness matrix corresponding to linear finite
  element discretization of \eqref{Model0} with boundary conditions
  given by~\eqref{neumman}. Then $A_+$ is an $M$-matrix relative
  of $A$ which is spectrally equivalent to $A$. The constants of equivalence depend only on the shape regularity of the mesh. Moreover, the graph
  corresponding to $A_+$ is connected.
\end{theorem}
\begin{proof}
The goal is to show that 
\[
\|u\|_A^2\le \|u\|_{A_+}^2\lesssim \|u\|_A^2,
\]
where the constants hidden in $\lesssim$ depend only on the shape regularity of
the mesh. 

The lower bound is clear, by just comparing~\eqref{auv-do}
and~\eqref{diag-compensate} As we discussed earlier in
Lemma~\ref{MMrel} such inequality shows that the graph corresponding
to $A_+$ is connected. Indeed, since $\|u\|_{A}^2$ vanishes only for
$u=(1,\ldots,1)^t$ it follows that $\|u\|_{A_+}^2$ also vanishes only
for $u=(1,\ldots,1)^t$ which proves that $A_+$ has only one connected
component.

To prove the upper bound, we fix an element $T$ and consider the local
stiffness matrix $A_T$ given by
\[
(A_T u,v) = b_T(u,v) =  
\sum_{e\in \partial T}\omega_{e,T} (\delta_e u)(\delta_ev).
\]
Denote \(\mathcal{E}_{T} = \{(i,j)\;\big|\; 1\le i<j\le (d+1)\}\),
and let $\mathcal{E}_{T}^{\pm}$ corresponding to $A_T$ be defined in a
way analogous to the definition of $\mathcal{E}^{\pm}$ for $A$. 

It is immediate to see that in the notation of Lemma~\ref{MMrel} the
minimum nonzero eigenvalue and the maximum in modulus negative
coefficient $\omega_{-,T}$ satisfy:
\[
\lambda_T=\lambda_{\min}(A_T)\eqqsim h_T^{d-2}\overline a_T,\quad\mbox{and}\quad
\omega_{-,T}=\max_{e\in \mathcal{E}_T^-}|\omega_e|
\lesssim h_T^{d-2} \overline a_T.
\]  
These relations hold with constants independent of
the mesh size $h_T$, but dependent on the shape regularity of the
mesh. Let us consider the bilinear form
\[
b_{+,T}(u,v) = \sum_{e\in \mathcal{E}_{T}^+}\omega_{e,T} (\delta_e u)(\delta_ev).
\]
Lemma~\ref{MMrel} then implies that for every $T$, $b_{+,T}(u,v)$ is
spectrally equivalent to $b_{T}(u,v)$ and the constants of spectral
equivalence depend only on the shape regularity of the mesh.  Summing
over all elements and using this spectral equivalence then gives:
\begin{eqnarray*}
&& \sum_{T}\sum_{e\in \mathcal{E}_T^+} \omega_{e,T}(\delta_eu)^2
\lesssim  \sum_{T}\sum_{e\in \mathcal{E}_T} \omega_{e,T} (\delta_eu)^2=\|u\|_A^2.
\end{eqnarray*}
On the other hand we have
\begin{eqnarray*}
  \|u\|_{A_+}^2 & = & \sum_{e\in \mathcal{E}^+}\omega_e(\delta_eu)^2 = 
\sum_{e\in \mathcal{E}}\max\{0, \omega_{e}\} (\delta_eu)^2
 = \sum_{T}\max\left\{0,\sum_{e\in \mathcal{E}_T} \omega_{e,T}\right\} (\delta_eu)^2\\
& \le & 
\sum_{T}\sum_{e\in \mathcal{E}_T} \max\left\{0,\omega_{e,T}\right\} (\delta_eu)^2
 = 
\sum_{T}\sum_{e\in \mathcal{E}_T^+} \omega_{e,T} (\delta_eu)^2.
\end{eqnarray*}
Combining these inequalities complete the proof.

\end{proof}
A simple corollary which we use later in proving estimates on the convergence rate 
is as follows.
\begin{corollary}\label{corollary-diag}
  Assume that $A$ is the stiffness matrix for piece-wise linear
  discretization of equation~\eqref{auv} and $A_+$ is the $M$-matrix
  relative defined in Theorem~\ref{m-matrix-plus}. Then the diagonal
  $D$ of $A$ and the diagonal $D_+$ of $A_+$ are spectrally
  equivalent.
\end{corollary}
\begin{proof}
For the
  diagonal elements of $A$ and $A_+$ we have
\[
[D]_{jj} = (\phi_j,\phi_j)_A \eqqsim (\phi_j,\phi_j)_{A_+} =[D_+]_{jj}.
\]
The equivalence ``$\eqqsim$'' written above follows from Lemma~\ref{m-matrix-plus}. 
\end{proof}

Corollary \ref{corollary-diag} together with Lemma~\ref{lemma-equiv}
provide the theoretical foundation for utilizing $M$-matrix relatives
in the design of AMG methods for linear systems with finite element
matrices.

\section{A general approach to the construction of coarse spaces}\label{sec:unifiedAMG}
We assume there exists a sequence of spaces
    $V_1, V_2, \dots, V_J$, which are not necessarily subspaces of
    $V$, but each of them is related to the original space $V$ by a
    linear operator 
\begin{equation} 
    \Pi_j: V_j\mapsto V.  
\end{equation} 
Our very basic assumption is that the following decomposition holds:
\begin{equation*}
    V=\sum_{j=1}^J\Pi_jV_j.
\end{equation*}
This means that for any $v\in V$, there exists $v_j\in V_j$ (which may not be unique) such that 
\begin{equation*}
    v=\sum_{j=1}^J\Pi_jv_j.
\end{equation*}
Denote
$$\utilde{W} = V_1\times V_2\times...\times V_J,
$$ 
with the inner product 
$$(\utilde u,\utilde v) = \sum_{i=1}^J(u_i,v_i),
$$ 
where $\utilde u=(u_1,...,u_J)^T$ and $\utilde
v=(v_1,...,v_J)^T$.  Or more generally, for $\utilde f=(f_1, \ldots, f_J)^T\in
\utilde V'$ with $f_i\in V_i'$, we can define
$$
( \utilde f, \utilde v) =\sum_{i=1}^J( f_i, v_i).
$$
We now define $\Pi_W:\utilde W\mapsto V$ by
$$
\Pi_W\utilde u = \sum_{i=1}^J \Pi_i u_i, \quad\forall \utilde
u=(u_1,...,u_J)^T\in \utilde W. 
$$ 
Formally, we can write 
$$
\Pi_W=(\Pi_1,\ldots,\Pi_J) \mbox{ and }
\Pi_W'=
\begin{pmatrix}
\Pi_1'\\
\vdots\\
\Pi_J'
\end{pmatrix}.
$$

We assume there is an operator $A_j: V_j\mapsto V_j'$ which is symmetric, positive semi-definite for each $j$ and define $\utilde A_W: \utilde W\mapsto \utilde W'$ as follows
  \begin{equation}\label{utildeAW}
      \utilde A_W: = \operatorname{diag}(A_1, A_2, \dots, A_J).
  \end{equation}

    For each $j$, we assume there is a symmetric positive
    definite operator $D_j: V_j\mapsto V_j'$, and define $\utilde D:
    \utilde W\mapsto \utilde W'$ as follows
  \begin{equation}\label{utildeD}
      \utilde D: = \operatorname{diag}(D_1, D_2, \dots, D_J).
  \end{equation}

    We associate a coarse space $V_j^c$, $V_j^c \subset V_j$, with
    each of the spaces $V_j$, and consider the corresponding
    orthogonal projection $Q_j: V_j\mapsto V_j^c$ with respect to
    $(\cdot, \cdot)_{D_j}$. We define $\utilde Q: \utilde W\mapsto \utilde W'$ by
    \begin{equation}
        \utilde Q: =\operatorname{diag}(Q_1, Q_2, \dots, Q_J).
    \end{equation}

\begin{assumption}\label{a:2level}\quad
  \begin{enumerate}
\item 
        The following inequality holds for all $\utilde w\in \utilde W$:
        \begin{equation}\label{assm:D_j}
            \|\Pi_W \utilde w\|_D^2\le C_{p,2}\|\utilde w\|_{\utilde D}^2, 
        \end{equation}
        for some positive constant $C_{p, 2}$.
\item For each $w\in V$, there exists a
  $\utilde w\in \utilde W$
  such that $w=\Pi_W\utilde w$ and the following inequality holds
    \begin{equation}\label{sum_Aj}
        \|\utilde w\|_{\utilde A_W}^2\le C_{p,1}\|w\|_A^2
\end{equation}
with a positive constant $C_{p, 1}$ independent of $w$. 
\item For all $j$, 
  \begin{equation}
    \label{NAjVjc}
N(A_j)\subset V_j^c.    
  \end{equation}

 \end{enumerate}
\end{assumption}

\begin{remark}
The above assumption implies that 
$$
w\in N(A) \Rightarrow \utilde w\in N(A_1)\times\ldots \times N(A_J).
$$
\end{remark}

We define the global coarse space $V_c$ by 
\begin{equation}\label{V_c}
    V_c:=\sum_{j=1}^J\Pi_j V_j^c.
\end{equation}

Further, for each coarse space $V_j^c$, we define
\begin{equation}\label{muj}
    \mu^{-1}_j(V_j^c):=\max_{v_j\in V_j}\min_{v_j^c\in V_j^c}\frac{\|v_j-v_j^c\|_{D_j}^2}{\|v_j\|_{A_j}^2},
\end{equation}
and 
\begin{equation}
  \label{muc}
    \mu_c=\min_{1\le j\le J}\mu_j(V_j^c),
\end{equation}
which is finite, thanks to Assumption \ref{a:2level}.3 (namely, \eqref{NAjVjc}).

By the two level convergence theory, if $D_j$ provides a convergent smoother, then
$(1-\mu_j(V_j^c))$ is the convergence rate for
two-level AMG method for $V_j$ with coarse space $V_j^c$. Next theorem
 gives an estimate on the convergence of the two level method in
terms of the constants from Assumptions~\ref{a:2level}
and $\mu_c$.
\begin{theorem}\label{thm:two-level}
If Assumption~\ref{a:2level} holds, then for each $v\in V$,  
    we have the following error estimate
    \begin{equation}
        \min_{v_c\in V_c}\|v-v_c\|_D^2 \le C_{p,1}C_{p,2}\mu_c^{-1}\|v\|_A^2.
    \end{equation}
\end{theorem}

    \begin{proof}
        By Assumption~\ref{a:2level}, for each $v\in V$, there exists $\utilde v\in \utilde V$ such that 
        \begin{equation}
            v=\Pi_W\utilde v
        \end{equation}
        and \eqref{sum_Aj} is satisfied.

        By the definition of $\mu_c$, we have
        \begin{equation}
            \|\utilde v-\utilde Q\utilde v\|_{\utilde D}^2\le \mu_c^{-1} \|\utilde v\|_{\utilde A_W}^2.
        \end{equation}
        We let $v_c=\Pi_W\utilde Q\utilde v$. Then $v_c\in V_c$ and by
        Assumption~\ref{a:2level}, we
        have
        \begin{equation*}
            \|v-v_c\|_D^2 =  \|\Pi_W(\utilde v- \utilde Q\utilde v)\|_D^2 \le  C_{p,2}\|\utilde v-\utilde Q\utilde v\|_{\utilde D}^2 \le  C_{p,2}\mu_c^{-1}\|\utilde v\|_{\utilde A_W}^2\le  C_{p,1}C_{p,2}\mu_c^{-1}\|v\|_A^2.
        \end{equation*}
        
      \end{proof}

We define another product space
\begin{equation}
    \utilde{V}:= V_c\times V_1\times V_2\times \cdots \times V_J,
\end{equation}
and we set $\Pi_c: V_c\mapsto V$ to be the natural inclusion from $V_c$ to $V$. Then we define $\Pi: \utilde V\mapsto V$ by
\begin{equation}
    \Pi:=(\Pi_c~\Pi_1~\Pi_2~\cdots~\Pi_J),
\end{equation}
and $\utilde A: \utilde V\mapsto \utilde V'$ by
\begin{equation}
  \utilde A_: = \begin{pmatrix}
      A_c & & & \\
      & A_1 & & \\
      & & \ddots & \\
      & & & A_J \\
      \end{pmatrix},
\end{equation}
where $A_c: V_c\mapsto V_c'$ is given as
\begin{equation}
    A_c: =\Pi_c'A\Pi_c.
\end{equation}
And $\utilde B: \utilde V\mapsto \utilde V'$ is given as
\begin{equation}
    \utilde B: = \begin{pmatrix}
      A_c^{-1} & & & &\\
      & D_1^{-1} & & &\\
      & & D_2^{-1} & &\\
      & & & \ddots &\\
      & & & & D_J^{-1}\\
      \end{pmatrix},
\end{equation}

We introduce the additive preconditioner $\widehat B$ 
\begin{equation}\label{additive_B}
    \widehat B: = \Pi \utilde B\Pi' = \Pi_cA_{c}^{-1}\Pi_c' + \sum_{j=1}^J \Pi_j D_j^{-1}\Pi_j',
\end{equation}
and we have the following lemma.

\begin{lemma}\label{lem:tBA0}
    If Assumption~\ref{a:2level} holds, then for any $v\in V$, there exists $\utilde v\in \utilde V$ such that  
    $$\|\utilde v\|_{\utilde B^{-1}}\le \tilde \mu_0\|v\|_A$$
    with $\tilde \mu_0$ being a constant depending on $C_{p,1}$, $C_{p,2}$, $\mu_c$ and $c^D$. 
\end{lemma}
\begin{proof}
    By Assumption~\ref{a:2level}, for each $v\in V$, there exists $\utilde w=(v_1~\cdots v_J)^T\in \utilde W$ such that $v=\Pi_W\utilde w$ and \eqref{sum_Aj} holds, namely,
    $$\|\utilde w\|_{\utilde A_W}^2\le C_{p, 1}\|v\|_A^2.$$
    We then define $\utilde v\in \utilde V$ by
    \begin{equation}
        \utilde v:=\begin{pmatrix}
            v_c\\
            \utilde w - \utilde Q \utilde w\\
            \end{pmatrix},
    \end{equation}
    with $v_c:= \Pi_W\utilde Q \utilde w$. Obviously, we have $\Pi\utilde v =v$, and $v_c$ satisfies
    \begin{equation}\label{Vc_approx}
        \|v-v_c\|_D^2 \le C_{p,1}C_{p,2}\mu_c^{-1}\|v\|_A^2.
    \end{equation}
    By Theorem~\ref{thm:two-level}, we have
    \begin{eqnarray*}
        \|v_c\|_A^2 &\le & 2\|v-v_c\|_A^2 +2\|v\|_A^2\\
        &\le & 2\|v-v_c\|_{\bar R^{-1}}^2 + 2\|v\|_A^2\\
        &\le & 2c^D\|v-v_c\|_D^2+2\|v\|_A^2\\
        &\le & 2c^DC_{p,1}C_{p,2}\mu_c^{-1}\|v\|_A^2+2\|v\|_A^2.
    \end{eqnarray*}
    Then we have
    \begin{eqnarray*}
        (\utilde B^{-1} \utilde v, \utilde v) & = & \|\utilde  w - \utilde Q\utilde w\|_{\utilde D}^2 +\|v_c\|_A^2\\
        & \le & \mu_c^{-1}\|\Pi_W\utilde w\|_A^2 +\|v_c\|_A^2\\
        & \le & C_{p,1}\mu_c^{-1}\|v\|_A^2 + 2(c^DC_{p,1}C_{p,2}+1)\|v\|_A^2 \\
        & = & (C_{p,1}\mu_c^{-1} + 2c^DC_{p,1}C_{p,2}+2)\|v\|_A^2 \\
    \end{eqnarray*}
\end{proof}

\begin{lemma}\label{lem:tBA1}
    If Assumption~\ref{assm:D_j} holds, then the following inequality holds for all $\utilde v\in \utilde V$
    $$\|\Pi\utilde v\|_A\le \tilde \mu_1\|\utilde v\|_{\utilde B^{-1}}, $$
    with constant $\tilde \mu_1$ depends on $C_{p,2}$ and $c^D$.
\end{lemma}
\begin{proof}
    For any decomposition $v=\Pi\utilde v = \Pi_W \utilde w +v_c$, we have
    \begin{eqnarray*}
        (\utilde B^{-1} \utilde v, \utilde v) & = & \|\utilde w\|_{\utilde D_W}^2 +\|v_c\|_A^2 \ge  \frac{1}{C_{p,2}}\|\Pi_W \utilde w\|_D^2 +\|v_c\|_A^2\\
        & = & \frac{1}{C_{p,2}}\|v-v_c\|_D^2 +\|v_c\|_A^2 \ge  \frac{1}{c^DC_{p,2}}\|v-v_c\|_{\bar{R}^{-1}}^2 +\|v_c\|_A^2\\
        & \ge & \frac{1}{c^DC_{p,2}}\|v-v_c\|_A^2 +\|v_c\|_A^2 \ge C(c^D, C_{p,2})\|v\|_A^2.
    \end{eqnarray*}
\end{proof}

Combining Lemma~\ref{lem:tBA0} and Lemma~\ref{lem:tBA1}, we
immediately have the following bound on the condition number of the
preconditioned system.
\begin{theorem}\label{thm:additive-converge}
    If Assumption~\ref{a:2level} holds, then 
    \begin{equation}
        \kappa (\widehat B A) \le \left(\frac{\tilde \mu_1}{\tilde \mu_0}\right)^2.
    \end{equation}
\end{theorem}

The following two-level convergence result is an application of 
the convergence theorem (Theorem~\ref{thm:two-level-convergence}) with
the error estimate in Theorem~\ref{thm:two-level}.

\begin{theorem}\label{thm:two-level-conv}
  If Assumption~\ref{a:2level} holds.
  Then the two-level AMG method with coarse space defined in
    \eqref{V_c} converges with a rate
    \[
        \|E\|_A^2\le 1-\frac{\mu_c}{C_{p,1}C_{p,2}c^D}.
    \]
\end{theorem}

\section{Classical AMG and jump coefficient problems}

In this section we consider an the Classical AMG method when applied
to a problem with heterogenous (jump) coefficients, namely
\eqref{Model0} with \eqref{coeff12}.  We begin with a discussion on
how the strength of connection is used to define the sparsity pattern
of the prolongation.

The strength of connection measure was introduced to handle cases such
as jump coefficients and anisotropies in the matrices corresponding to
discretizations of scalar PDEs.  An important observation regarding
the classical AMG is that the prolongation matrix $P$, which defines
the basis in the coarse space, uses only strong connections.  

To begin with, we first introduce a strength operator as follows
\begin{equation}
    s_c(i, j)=\frac{a_{ij}}{\max\bigg(\min_{k\neq i}a_{ik}, \min_{k\neq j}a_{jk}\bigg)}, \quad  1\le i, j\le n.
\end{equation}
The definition above is symmetrized version of strength function used in the classical AMG literature.

Given a threshold $\theta\in (0, 1)$, we define the strength operator 
\begin{equation}
    S=\sum_{s_c(i, j) > \theta} e_ie_j^T,
\end{equation}
and a filtered matrix $A_S: V\mapsto V'$
\begin{equation}
    (A_Su, v) = \sum_{e=(i, j), S_{ij}\ne 0}\omega_e\delta_eu\delta_ev.
\end{equation}

We have the following lemma
\begin{lemma}
    $A_S$ is an $M$-matrix relative of $A$.
\end{lemma}
\begin{proof}
    We recall the definition of $A_{+}$ in \eqref{diag-compensate}. By Theorem~\ref{m-matrix-plus}, we immediately have 
    \begin{equation}
        \|v\|_{A_S}\le \|v\|_{A_+}\lesssim \|v\|_A, \quad \forall v\in V.
    \end{equation}
    Let $D_S$ be the diagonal of $A_S$ and we denote the $i$-th diagonal entries of $D_S$ and $D$ by $\tilde d_i$ and $d_i$ respectively. Then, by the definition of the strength of 
connection, we have
    \begin{eqnarray*}
        \tilde d_i &=& \sum_{j\in N_i, s_c(i, j)\ge \theta}\omega_{ij} \ge \theta\sum_{j\in N_i, s_c(i, j)\ge \theta} \max_{j\in N_i} \omega_{ij} \ge \frac{\theta}{|N_i|}\sum_{j\in N_i, s_c(i, j)\ge \theta} d_i\ge \frac{\theta}{|N_i|} d_i.
    \end{eqnarray*}
    This gives us
    \begin{equation}
        \|v\|_D^2\le \frac{\max_{i}|N_i|}{\theta}\|v\|_{D_S}^2, \quad \forall v\in V.
    \end{equation}
    This completes the proof.
\end{proof}

Thanks to the results in \S\ref{sec:m-matrix} and \S\ref{sec:m-matrix-fem}, without loss of generality, we assume that $A$ is an $M$-matrix with all connections being strong connections. 

We then use an MIS algorithm to identify $\mathcal C$, the set of coarse points, to form a
$C/F$-splitting, and
\begin{equation*}
    \mathcal C\bigcup \mathcal F = \Omega:=\{1, \dots, n\}, \quad \mathcal C \bigcap \mathcal F\neq\emptyset.
\end{equation*}
For convenience, we reorder the indices so that $\mathcal C=\{1, \dots, J\}$.
\begin{equation}\label{Omegaj}
\Omega=\bigcup_{j=1}^J\Omega_j.
\end{equation} 
where $\Omega_j$ is defined for each $j\in \mathcal C$ as follows 
\begin{equation}\label{CAMG_Omega}
    \Omega_j:=\{j\}\bigcup F_j^s, \quad j=1, \dots, J.
\end{equation}
Here $F_j^s:=\mathcal F\bigcap s_j$, and $s_j$ is the set of interpolation neighbors
of $j$. This depends on the choice of interpolation. For example, in the direct interpolation we introduced in~\cite{1stAMG,Trottenberg.U;Oosterlee.C;Schuller.A.2001a}, $s_j$ is $N_j$, the set of neighbors of $j$; in the standard interpolation~\cite{1stAMG,Trottenberg.U;Oosterlee.C;Schuller.A.2001a},
\begin{equation}
    s_j=N_j\bigcup \left(\bigcup_{i\in N_j}N_i\right).
\end{equation}

In the discussion follows, we choose the standard
interpolation,
since the extension to other interpolations is trivial.

For each $\Omega_j$ we denote 
\begin{equation}
    \Omega_j=\{m_1, m_2, \dots, m_{n_j}\},
\end{equation}
and let $n_j:=|\Omega_j|$, namely, $n_j$ is the cardinality of $\Omega_j$. 
In accordance with the notation in~\S\ref{sec:unifiedAMG}. We then define 
\begin{equation}
    V_j:=\mathbb{R}^{n_j},
\end{equation}
and the associated operator $\Pi_j: V_j\mapsto V$ 
\begin{equation}\label{pi-camg}
 (\Pi_jv)_i = \begin{cases}
        p_{m_k,k}v_{k}, & \text{ if } i=m_k,\\
        0, & \text{ if } i\notin \Omega_j\\
    \end{cases},
\end{equation}
where $p_{m_k,k}$ are given weights.  As all the constructions below
will be based on the $M$-matrix relative of $A$, and without loss of
generality, we may just use $A$ to denote this $M$-matrix relative.

Following~\S\ref{sec:unifiedAMG}, we introduce the operator $\chi_j: V\mapsto V_j$:
\begin{equation}
    (\chi_jv)_i: = v_{m_i}.
\end{equation}
which takes as argument a vector $v$ and returns only the portion of
it with indices in $\Omega_j$, namely, $\chi_j v$, is a vector in
$\mathbb{R}^{n_j}$.  It is immediate to verify that
\begin{equation*}\label{partion-unity}
    \sum_{j=1}^J\Pi_j\chi_j= I.
\end{equation*}

The local operators $A_j: V_j\mapsto V_j'$ are defined as follows
    \begin{equation}\label{Aj}
        (A_ju, v)=\sum_{\substack{e\in \mathcal E\\e\subset \Omega_j}}\omega_e\delta_{j,e}u\delta_{j,e}v.
    \end{equation}
    Here, $e\subset \Omega_j$ means the two vertices connected by $e$ are in $\Omega_j$. Notice that $A_j$ is symmetric positive semi-definite.

\begin{lemma}
    For any $v\in V$, the following holds for $v_j=\chi_jv$
    \begin{equation}
        \sum_{j=1}^J\Pi_jv_j = v, \text{ and } \sum_{j=1}^J\|v_j\|_{A_j}^2\le \co\|v\|_A^2,
    \end{equation}
    where $\co$ is a constant depending on the overlaps in the partition $\{\Omega_j\}_{j=1}^J$
    \begin{equation}
      \co= \max_{1\le j \le J} \left|\{l\; : \; \Omega_l\cap \Omega_j \neq \emptyset\}\right|.
    \end{equation}
\end{lemma}
\begin{proof}
    By \eqref{partion-unity}, we have $\sum_{j=1}^J\Pi_jv_j =v$. By definitions
    \begin{equation}\label{verifyA}
          \sum_{j=1}^{m_c}\|v_j\|_{A_j}^2=\sum_{j=1}^{m_c}\sum_{\substack{e\in \mathcal E\\e\subset \Omega_j}}\omega_e(\delta_ev)^2\le \co\sum_{e\in \mathcal E}\omega_e(\delta_e v)^2 =\co\|v\|_{A}^2.
    \end{equation}
    This complets the proof.
\end{proof}
If $D$ is the diagonal of $A$, then we set $D_j$, $j=1:J$ to be the
restriction of $D$ on $\Omega_j$, namely, in
$\mathbb{R}^{n_j\times n_j}$ and
\begin{equation}\label{local_diag}
    (D_j)_{ii}=D_{m_i, m_i}, \quad\mbox{or equivalently}\quad D_j=\chi_jD_j\chi_j'
\end{equation}
We have the following lemma which shows \eqref{assm:D_j}.
\begin{lemma}\label{lem:Dj}
For $D_j$ defined in \eqref{local_diag}, the following inequality holds
    \begin{equation}
        \|\sum_{j=1}^{m_c}\Pi_jv_j\|_D^2\le \co\sum_{j=1}^{m_c}\|v_j\|_{D_j}^2, \quad \forall v_j\in V_j.
    \end{equation}
\end{lemma}
\begin{proof}
    Recall from the definition of $\Pi_j$, we have 
    \begin{equation}
        \|\Pi_jv_j\|_D\le \|v_j\|_{D_j}, \quad \forall v_j\in V_j.
    \end{equation}
Therefore, 
    \begin{eqnarray*}
        \|\sum_{j=1}^{J}\Pi_jv_j\|_D^2 &= & \left(D\sum_{i=1}^{J}\Pi_iv_i, \sum_{j=1}^{J}\Pi_jv_j\right)=  \sum_{i=1}^{J}\sum_{j=1}^{J}(D\Pi_iv_i, \Pi_jv_j)\\
                  &= & \sum_{\substack{1\le i, j\le J\\\Omega_i\cap\Omega_j\ne \emptyset}}(D\Pi_iv_i, \Pi_jv_j)\le  \sum_{\substack{1\le i, j\le J\\\Omega_i\cap\Omega_j\ne \emptyset}}\frac{\|\Pi_jv_i\|_D^2+\|\Pi_jv_j\|_D^2}{2}\\
                  &\le & \co\sum_{j=1}^{J}\|v_j\|_{D_j}^2. 
    \end{eqnarray*}
\end{proof}
We choose the local coarse spaces $V_j^c$ as
\begin{equation}
    V_j^c:=\operatorname{span}\{\bm 1_{n_j}\},\quad \bm{1}_{n_j}=(\underbrace{1,1,\ldots,1}_{n_j})^T
\end{equation}
Then by definition, we have
\begin{equation}
    \mu_j(V_j^c)=\lambda_j^{(2)},
\end{equation}
where $\lambda_j^{(2)}$ is the second smallest eigenvalue of the matrix $D_j^{-1}A_j$.
The global coarse space $V_c$ is then obtained by~\eqref{V_c}, and is
\begin{equation}
    V_c=\operatorname{span}\{P_1, P_2, \cdots, P_J\}.
\end{equation}
Finally, by Theorem~\ref{thm:two-level-convergence}, the converges rate of this
two-level geometric multigrid method depends on the
$\min_j(\lambda_j^{(2)})$. If the discrete Poincar{\'e } inequality is
true for each $V_j$, namely ,
\begin{equation}\label{local_poincare}
    \inf_{v_c\in V_j^c} \|v-v_c\|_{D_j}^2\le c_j\|v\|_{A_j}^2, \quad \forall v\in V_j,
\end{equation}
with $c_j$ to be a constant, then the two-level classical AMG method converges uniformly.

We now consider the convergence of classical two-level AMG with
standard interpolation for the jump coefficient problem and we prove
a uniform convergence result for the two level method. 
Before we go through the AMG two-level convergence proof, we first
introduce the following result on a connected graph, which can be viewed
as a discrete version of Poincar\'e inequality.
\begin{lemma}\label{lem:discrete_poincare}
    We consider the following graph Laplacian on a connected undirected graph $\mathcal{G}=(\mathcal{V}, \mathcal{E})$ 
    \begin{equation}\label{graph_lapla}
        \langle Au, v\rangle = \frac{1}{2}\sum_{(i,j)\in \mathcal{E}}(u_i-u_j)(v_i-v_j).
    \end{equation}
    For any $v\in V$, the following estimate is true
    \begin{equation}
        \|v-v_c\|_{\ell^2}^2\le \mu n^2d\langle Av, v\rangle,
    \end{equation}
    where $n=|\mathcal V|$ is the size of the graph, $v_c=\sum_{j=1}^{n}w_jv_j$ is a weighted average of $v$, $\mu=\sum_{j=1}^nw_j^2$, and $d$ is the diameter of the graph.
\end{lemma}
\begin{proof}
    Since $\mathcal G$ is connected, we have, for each pair of vertices $i$ and $j$, there exist $l\le d$ and a path $k_0 \rightarrow k_1 \rightarrow \cdots \rightarrow k_l$ with $k_0=i$ and $k_l = j$ such that $(k_{m-1}, k_m)\in \mathcal E$, $\forall m= 1, \dots, l$. We then have
    \begin{equation*}
        (v_i-v_j)^2= \left(\sum_{m=1}^l (v_{k_{m-1}}- v_{k_m})\right)^2\le l\sum_{m=1}^l(v_{k_{m-1}} - v_{k_m})^2 \le d \langle Av, v\rangle.
    \end{equation*}
   Combining this with Cauchy-Schwarz inequality, we obtain 
    \begin{eqnarray*}
        \|v-v_c\|_{\ell^2}^2 & = & \sum_{i=1}^n\left(v_i-\sum_{j=1}^nw_jv_j\right)^2 = \sum_{i=1}^n\left(\sum_{j=1}^nw_j(v_i-v_j)\right)^2\\
        & \le & \mu\sum_{i=1}^n\sum_{j=1}^n(v_i-v_j)^2 \le \mu n^2 d\langle Av, v\rangle.
    \end{eqnarray*}
\end{proof}

Next Lemma is a spectral equivalence result, showing that the local
operators $A_j$, defined in~\eqref{Aj}, for shape regular mesh, are
spectrally equivalent to a scaling of the graph Laplacian operators
$A_{L,j}$ defined as
\begin{equation}\label{e:ALj}
(A_{L,j} u,v) = \frac12 \sum_{(i,k)\in \Omega_j} (u_i-u_k)(v_i-v_k). 
\end{equation}
\begin{lemma}\label{lem:shape_reg}
  With the assumption we made on the shape regularity of the finite
  element mesh, the following inequalities hold for $A_j$ defined as
  in \eqref{Aj} using the standard interpolation
            \begin{equation}
                c_Lh^{d-2}\langle A_{L,j} v_j, v_j\rangle \le ( A_jv_j, v_j) \le c^Lh^{d-2} \langle A_{L,j} v_j, v_j\rangle,
            \end{equation}
            where $A_{L,j}$ is a graph Laplacian defined in \eqref{graph_lapla} on the graph $\mathcal G_j$, $h$ is the mesh size and $c_L$, $c^L$ are constants depend on the shape regularity constant, and the threshold $\theta$ for the strength of connections. 
    \end{lemma}

\begin{proof}
  By the definition of the strength of connection, we have
    \[
        a_{ii} =\sum_{k\in N_k}-a_{ik} \le -\frac{|N_i|}{\theta}a_{ij},
    \]
    Since $A$ is symmetric, we also have
    \[
        a_{ii}\le -\frac{|N_i|}{\theta}a_{ji},
    \]
    By the definition of $\Omega_j$ in standard interpolation, for any $i\in \Omega_j\setminus\{j\}$, either $i\in F_j^s$ or there exists a $i\in F_j^s$ such $i\in F_{k}^s$. For the latter, $(j, k, i)$ forms a path between $j$ and $i$ going along strong connections. We have then 
    \[
        -a_{ik} \ge  -\frac{\theta}{|N_k|} a_{kk} \ge -\frac{\theta}{|N_k|} a_{kj}\ge -\frac{\theta^2}{|N_k||N_j|} a_{jj}.
    \]
    and 
    \[
        a_{jj}\ge -a_{jk} \ge \frac{|N_k|}{\theta}a_{kk} \ge -\left(\frac{|N_k|}{\theta}\right)^2a_{ik} 
    \]
    Combining the above two inequalities and using the assumption that the mesh is shape regular, for any $l\in \Omega_j$ that is connected with $i$ we have 
    \[
        \sigma_1 a_{ij} \le -a_{il} \le \sigma_2 a_{jj}
        \]
    with constants $\sigma_1$ and $\sigma_2$ which depend on the shape regularity constant and $\theta$.

    Since in the definition of $A_j$ in \eqref{Aj}, $\omega_e=-a_{ij}/2$ for $e=(i, j)$ , we obtain
\begin{equation}
    c_1 a_{jj}\langle A_{L,j} v_j, v_j\rangle \le ( A_jv_j, v_j) \le c_2 a_{jj}\langle A_{L,j} v_j, v_j\rangle.
\end{equation}
    Then by a scaling argument, $a_{jj}\eqqsim h^{d-2}$ and the proof is complete. 
\end{proof}

\begin{theorem}\label{thm:two-level-jump}
  The two level method using a coarse space defined as $V_c$ defined via the classical AMG
  is uniformly convergent.
\end{theorem}
\begin{proof}
    By Theorem~\ref{thm:two-level-convergence}, we only need to show that $\mu_c$ is bounded, which can be easily obtained by combining Lemma~\ref{lem:discrete_poincare} -- \ref{lem:shape_reg} with Lemma~\ref{lemma-equiv}.
\end{proof}

We point out that Theorem~\ref{thm:two-level-jump} is also true for two level unsmoothed aggregation AMG. The proof is identical to the proof for classical AMG case.

\section*{Acknowledgments}
The work of Xu was partially supported by the DOE Grant
DE-SC0009249 as part of the Collaboratory on Mathematics for
Mesoscopic Modeling of Materials and by NSF grants DMS-1412005
and DMS-1522615.  The work of Zikatanov was partially supported by
by NSF grants DMS-1418843 and DMS-1522615.

\bibliographystyle{unsrt}
\bibliography{refs}

\end{document}